\documentclass[12pt,reqno]{amsart}

\usepackage[latin1]{inputenc}
\usepackage{amsmath}
\usepackage{amsfonts}
\usepackage{amssymb}
\usepackage{graphics}
\usepackage{enumerate}
\usepackage{subfigure}
\usepackage{amssymb,amsmath,amsthm,amscd,epsf,latexsym,verbatim,graphicx,amsfonts}
\input epsf.tex

\usepackage{amscd}
\usepackage{amsmath}
\usepackage{amssymb}
\usepackage{amsthm}
\usepackage{epsf}
\usepackage{latexsym}
\usepackage{verbatim}
\input epsf.tex

\theoremstyle{plain}
\newtheorem{theorem}{Theorem}[section]

\newtheorem{corollary}[theorem]{Corollary}
\newtheorem{lemma}[theorem]{Lemma}
\newtheorem{prop}[theorem]{Proposition}
\theoremstyle{definition}

\theoremstyle{remark}

\newcommand{\nri}{n\rightarrow\infty}
\newcommand{\bbZ}{\mathbb{Z}}

\newcommand{\bbC}{\mathbb{C}}

\newcommand{\bbN}{\mathbb{N}}

\newcommand{\mcn}{\mathcal{N}}
\newcommand{\mcg}{\mathcal{G}}
\newcommand{\mcp}{\mathcal{P}}
\newcommand{\mcl}{\mathcal{L}}
\newcommand{\mcm}{\mathcal{M}}
\newcommand{\mcr}{\mathcal{R}}

\newcommand{\eitheta}{e^{i\theta}}

\newcommand{\at}{\makeatletter @\makeatother}

\DeclareMathOperator*{\supp}{supp}

\DeclareMathOperator*{\wlim}{w-lim}

\DeclareMathOperator*{\spn}{span}
\DeclareMathOperator{\dt}{det}

\title[]{The Bergman Shift Operator on Polynomial Lemniscates}
\author[]{Brian Simanek}
\date{}

\begin{document}
\maketitle

\begin{abstract}
We investigate the relationship between the Bergman shift operator and the support of the corresponding measure.  We pay special attention to the situation when the measure of orthogonality is concentrated on a polynomial lemniscate.  As an application of our new results, we obtain a ratio asymptotic result for a wide variety of measures supported on polynomial lemniscates.
\end{abstract}

\vspace{4mm}


\vspace{2mm}


\vspace{2mm}

\normalsize

\section{Introduction}\label{intro}

Let $\mu$ be a finite Borel measure whose support is an infinite and compact subset of the complex plane $\bbC$.  Given such a measure, one can perform Gram-Schmidt orthogonalization on the sequence $\{1,z,z^2,z^3,\ldots\}$ in the space $L^2(\bbC,\mu)$ to arrive at a sequence of polynomials $\{\varphi_n(z;\mu)\}_{n\geq0}$ satisfying
\[
\int_{\bbC}\varphi_n(z;\mu)\overline{\varphi_m(z;\mu)}d\mu(z)=\delta_{m,n},
\]
and normalized so that $\varphi_n$ has positive leading coefficient $\kappa_n=\kappa_n(\mu)$.  The polynomials $\{\varphi_n\}_{n=0}^{\infty}$ are called the orthonormal polynomials for the measure $\mu$.  We will also denote by $\Phi_n(z;\mu)=\kappa_n^{-1}\varphi_n(z;\mu)$ the monic orthogonal polynomial of degree $n$.

Our focus here will be on the relationship between the measure $\mu$ and the associated \textit{Bergman shift operator}.  To define this operator, let $\mcm:L^2(\mu)\rightarrow L^2(\mu)$ be the map given by $(\mcm f)(z)=zf(z)$.  Let $\mcp$ be the closure of the span of the polynomials inside $L^2(\mu)$.  It is easy to see that $\mcm$ maps $\mcp$ to itself and that the restricted map has matrix representation
\[
M=\begin{pmatrix}
M_{11} & M_{12} & M_{13} & M_{14} & \cdots\\
M_{21} & M_{22} & M_{23} & M_{24} & \cdots\\
0 & M_{32} & M_{33} & M_{34} & \cdots\\
0 & 0 & M_{43} & M_{44} & \cdots\\
0 & 0 & 0 & M_{54} & \cdots\\
\vdots & \vdots & \vdots & \vdots & \ddots
\end{pmatrix}
\]
in terms of the orthonormal basis given by the orthonormal polynomials (see \cite{Shift}).  Indeed, since $M_{jk}=\langle z\varphi_{k-1},\varphi_{j-1}\rangle$, it is easy to see that $M_{jk}=0$ if $j>k+1$.  Thus, the operator $\mcm$ determines a Hessenberg matrix $M$.  The relationship between $\Phi_n$ and the matrix $M$ is Proposition \ref{det} (below), which is contained in \cite[Theorem 6.2]{CDRev}.  To properly state it, we define $\dt_n$ to be the determinant of a linear operator from $\spn\{1,z,\ldots,z^{n-1}\}$ to itself.

\begin{prop}\label{det}
Let $\pi_n$ be the projection onto the $n$-dimensional subspace given by the span of $\{1,z,\ldots,z^{n-1}\}$ inside $\mcp$.  The polynomial $\Phi_n(z;\mu)$ and the matrix $M$ are related by
\begin{align}\label{keydet}
\Phi_n(z;\mu)=\dt_n\left(z-\pi_n M\pi_n\right).
\end{align}
\end{prop}

Given a measure $\mu$, one can calculate the corresponding matrix $M$ by evaluating the necessary integrals.  This is not always practical, so it is common to look for properties of the measure $\mu$ that manifest themselves in the matrix $M$.  For example, it is known if the measure $\mu$ is supported on a compact subset of the real line, then the corresponding matrix $M$ is self-adjoint and is non-zero only on its three main diagonals.  A more non-trivial result is contained in \cite{Shift}.  If $\mcr$ denotes the right shift operator on $\ell^2(\bbN)$ and $\mcl$ denotes the left shift operator, then we say a matrix $A$ is \textit{weakly asymptotically Toeplitz} (see \cite{SimHess1}) if the sequence
\[
\{\mcl^nA\mcr^n\}_{n\in\bbN}
\]
converges weakly to a Toeplitz matrix as $\nri$.  With this notation, the relevant result from \cite{Shift} can be stated as:

\begin{theorem}[Saff $\&$ Stylianopoulos, \cite{Shift}]\label{SS1}
If $\mu$ is area measure on a region $G$ whose boundary is a Jordan curve that is piecewise analytic without cusps, then the corresponding matrix $M$ is weakly asymptotically Toeplitz.
\end{theorem}

\noindent\textit{Remark.}  The result proved in \cite{Shift} is stronger in that it provides the symbol of the limiting Toeplitz operator and estimates on the rate of convergence.

\vspace{2mm}

\noindent\textit{Remark.}  An extension of Theorem \ref{SS1} to more general measures was proven by the author in \cite{SimHess1}.

\vspace{2mm}

It is also interesting to consider the problem inverse to the one just discussed.  More precisely, it is natural to ask what properties of the measure $\mu$ can be deduced from certain knowledge of the matrix $M$.  There is an extensive literature on this subject when the measure $\mu$ is supported on the unit circle $\{z:|z|=1\}$ or a compact subset of the real line (see for example \cite{OPUC1,OPUC2,Rice}).  These classical cases are special because the entries of the matrix $M$ can be described in terms of the coefficients appearing in a recurrence relation satisfied by the orthonormal polynomials.  In more general settings, no such recurrence relation exists (see \cite{LBS,PuSty}), but the inverse problem is still an interesting one.

Among the most basic properties of a measure is its support.  Therefore, we can begin our investigation by attempting to relate properties of the matrix $M$ to conditions satisfied by the support of the measure $\mu$.  However, we must proceed cautiously since some observations can be misleading.  Consider the following example.

\vspace{2mm}

\noindent\textbf{Example.}  Let $\mu$ be area measure on the unit disk $\{z:|z|<1\}$.  It is easy to verify that the matrix $M$ is weakly asymptotically Toeplitz and the limiting matrix is just the right shift operator on $\ell^2(\bbN)$.  This means that for large values of $n$, one has $\|z\Phi_n(z;\mu)\|\sim\|\Phi_{n}(z;\mu)\|$, which suggests that the measure is supported - or at least heavily concentrated - near the set $\{z:|z|=1\}$.  However, this is not the case since the measure $\mu$ is evenly distributed across the whole unit disk.

\vspace{2mm}

The above example highlights the main difficulty in using the matrix $M$ to identify the support of the measure $\mu$.  It is very often the case that the behavior of the measure near the boundary of the polynomial convex hull of its support is the dominant influence in determining the matrix $M$ (see for example \cite{Slides}).  Therefore, it is easy to find measures with substantially different supports whose $M$ matrices are very nearly identical.

However, there is a way to state precisely what portions of a measure most heavily influence the behavior of the matrix $M$.  This is done by investigating the weak limits of the sequence $\{|\varphi_n(z;\mu)|^2d\mu(z)\}_{n\geq0}$.  Any weak limit of this sequence is called a \textit{weak asymptotic measure}.  It is well-known that the polynomial $\Phi_n(z;\mu)$ satisfies
\[
\|\Phi_n(z;\mu)\|=\min\{\|Q\|:Q=z^n+\cdots\}.
\]
Consequently, one expects the polynomial $\Phi_n(z;\mu)$ to be small where the measure $\mu$ is dense and larger where the measure $\mu$ is sparse (to the extent this is possible).  Therefore, by considering the weak asymptotic measures, we can understand what parts of the measure $\mu$ exhibit the most influence in determining the asymptotic behavior of the orthonormal polynomials.  Consequently, we will shift our focus to relating properties of the matrix $M$ to the support of the weak asymptotic measures.

As stated above, this line of inquiry has a long history in the classical settings of the unit circle and the real line.  One of the most profound results in this context is the so-called \textit{Magic Formula} from \cite{MagicForm}.  We refer the reader to \cite{MagicForm} for a precise statement of the Magic Formula, but we will mention that it relates the matrix $P(M)$ for an appropriate polynomial $P$ to properties of the support of the corresponding measure.  Our main result - which takes a similar form - is the following:

\begin{theorem}\label{magicish}
Let $\mu$ be a finite measure with compact and infinite support in the complex plane and let $P(z)$ be a monic polynomial of degree $q\geq1$.  Fix $r>0$.  The matrices $\{(P(M)-r\mcr^q)\mcr^n\}_{n\in\bbN}$ converge strongly to $0$ as $\nri$ if and only if both of the following conditions are satisfied:
\begin{itemize}
\item[i)]  $\lim_{\nri}\kappa_n\kappa_{n+q}^{-1}=r$,
\item[ii)]  every weak asymptotic measure is supported on $\{z:|P(z)|=r\}$.
\end{itemize}
\end{theorem}

Theorem \ref{magicish} will allow us to deduce additional properties of the orthonormal polynomials in cases when the hypotheses of the theorem are satisfied.  Before we state our first corollary, we need to remind ourselves of the notion of a \textit{right limit} of a matrix.  For a matrix $A$, let $A_n^{(m)}$ be the $2m+1\times2m+1$ sub-matrix of $A$ centered at $A_{n,n}$.  We will say that a matrix $X$ is a right limit of the matrix $M$ if there is a subsequence $\mcn\subseteq\bbN$ so that for every $m\in\bbN$ it holds that
\[
\lim_{{\nri}\atop{n\in\mcn}}M_n^{(m)}=X_0^{(m)}.
\]
It is easy to see that right limits always exist whenever the matrix elements of $M$ are bounded (which they are in the cases we are considering) and that every right limit is a bi-infinite matrix even though $M$ is indexed by the natural numbers.  The subject of right limits in the context of Jacobi matrices is discussed in \cite[Chapter 7]{Rice} in terms of the recursion coefficients for orthogonal polynomials on the real line.  With this terminology, we can state our first corollary.

\begin{corollary}\label{prl}
Let $P$ be a monic polynomial of degree $q\geq1$ and fix $r>0$.  If the matrices $\{(P(M)-r\mcr^q)\mcr^n\}_{n\in\bbN}$ converge strongly to $0$ as $\nri$ then
\begin{itemize}
\item[a)]  For every $z$ not in the convex hull of the support of $\mu$ it is true that
\begin{align*}
\lim_{\nri}\frac{P(z)\varphi_n(z;\mu)}{r\varphi_{n+q}(z;\mu)}=1,
\end{align*}
\item[b)]  If $X$ is a right limit of the matrix $M$, then $P(X)=r\mcr^q$ and $X$ is $q$-periodic along its diagonals.
\end{itemize}
\end{corollary}

\begin{proof}[Proof of Corollary \ref{prl}]
To prove the ratio asymptotic statement, notice that our hypotheses and Theorem \ref{magicish} imply
\[
\lim_{\nri}\|r^{-1}P\varphi_n\|_{L^2(\mu)}=1,\qquad\lim_{\nri}\frac{\kappa_n}{r\kappa_{n+q}}=1.
\]
The desired conclusion now follows from \cite[Theorem 2.2]{SimaRat}.

Turning our attention to the right limits, notice that our hypotheses imply
\[
\lim_{\nri}P(M)_{n-j,n}=
\begin{cases}
0,& \qquad\, j\neq-q\\
 r, & \qquad\, j=-q.
\end{cases}
\]
Therefore, the unique right limit of the matrix $P(M)$ is $r\mcr^q$.  It follows easily from the Hessenberg structure of $M$ that if $X$ is any right limit of the matrix $M$, then $P(X)=r\mcr^q$.  We then see that
\[
X\mcr^q=\frac{1}{r}XP(X)=\frac{1}{r}P(X)X=\mcr^qX,
\]
so $X$ is $q$-periodic along its diagonals.
\end{proof}

We will explore some examples of measures satisfying the hypotheses of Corollary \ref{prl} in Section \ref{eg}.  In the meantime, let us explore some consequences of that result.  Notice that it enables us to deduce some properties of every right limit of the matrix $M$ using only elementary algebra.  The following result is one example of such a property.

\begin{corollary}\label{limittr}
Suppose that $P$ is a monic polynomial of degree $q\geq1$, that $r>0$, and that the matrices $\{(P(M)-r\mcr^q)\mcr^n\}_{n\in\bbN}$ converge strongly to $0$ as $\nri$.  If $P(x)=x^q-\alpha x^{q-1}+\cdots$, then
\[
\lim_{\nri}\sum_{j=1}^qM_{n+j,n+j}=\alpha.
\]
Equivalently, if $X$ is any right limit of $M$, then for every $n\in\bbN$
\[
\sum_{j=1}^qX_{n+j,n+j}=\alpha;
\]
specifically it is independent of $n$.
\end{corollary}

Given the results we have stated so far, it is natural to ask if there are examples of measures to which we can apply them.  Results concerning orthogonal polynomials with respect to measures supported on polynomial lemniscates can be found in \cite{Islands,Widom} and provide us with some examples, but very general results are difficult to find.  The following theorem provides a wealth of examples to which we can apply Theorem \ref{magicish} and the above corollaries.

\begin{theorem}\label{app}
Let $P$ be a monic polynomial of degree $q\geq1$ and let $r>0$ be fixed.  Suppose $\mu$ is a measure on $\{z:|P(z)|\leq r\}$ such that for each compact subset $K\subseteq\{z:|P(z)|<r\}$ there is a constant $C_K$ so that
\begin{align}\label{kbound}
\sum_{n=0}^{\infty}|\varphi_n(z;|P|^2\mu)|^2<C_K,\qquad z\in K.
\end{align}
Then the operators $\{(P(M)-r\mcr^q)\mcr^n\}_{n\in\bbN}$ converge strongly to zero as $\nri$.
\end{theorem}

\noindent\textit{Remark.}  The condition (\ref{kbound}) is used in a similar context in \cite{SSST} (see \cite[Lemma 2.1]{SSST}, and also \cite{Slides}).

\vspace{2mm}

A particularly interesting consequence of Theorem \ref{app} is the following:

\begin{corollary}\label{areap}
Let $P$ be a monic polynomial of degree $q\geq1$ and let $r>0$ be fixed. If $\mu$ is area measure on $\{z:|P(z)|\leq r\}$, then $\kappa_n\kappa_{n+q}^{-1}\rightarrow r$ as $\nri$.
\end{corollary}

If we define a \textit{$q$-block Toeplitz} matrix as a matrix that is $q$-periodic along its diagonals, then another way of stating the conclusion of Corollary \ref{prl}b is to say that every right limit of $M$ is $q$-block Toeplitz.  Let us write the matrix $M$ as
\[
M=\begin{pmatrix}
M^{(q)}_{11} & M^{(q)}_{12} & M^{(q)}_{13} &  \cdots\\
M^{(q)}_{21} & M^{(q)}_{22} & M^{(q)}_{23} &  \cdots\\
0 & M^{(q)}_{32} & M^{(q)}_{33} &  \cdots\\
0 & 0 & M^{(q)}_{43} &  \cdots\\
\vdots & \vdots & \vdots &  \ddots
\end{pmatrix}
\]
where each $M_{ij}^{(q)}$ is a $q\times q$ matrix.  It is clear that the property of being $q$-block Toeplitz is equivalent to this block matrix form of $M$ being constant along its diagonals.  This also motivates our defining the property of being \textit{asymptotically $q$-block Toeplitz}.  We will say that the matrix $M$ is asymptotically $q$-block Toeplitz if there is a $q$-block Toeplitz matrix $X$ so that
\[
\wlim_{\nri}\mcl^{qn}M\mcr^{qn}=X.
\]
Corollary \ref{prl}b makes the conclusion that every right limit is $q$-block Toeplitz, but this does not imply the matrix $M$ is asymptotically $q$-block Toeplitz..  Our next result provides necessary and sufficient conditions for the matrix $M$ to be asymptotically $q$-block Toeplitz.

\begin{theorem}\label{uniquer}
Let $\mu$ be a finite measure of compact and infinite support.  Assume also that $\liminf_{\nri}\kappa_{n-1}\kappa_{n}^{-1}>0$.  The corresponding matrix $M$ is asymptotically $q$-block Toeplitz if and only if for every $s\in\{0,1,\ldots,q-1\}$ there is a positive real number $R$ and a function $f_s$ that is analytic in $\{z:R<|z|\leq\infty\}$ so that
\[
\lim_{\nri}\frac{\varphi_{nq+s-1}(z;\mu)}{\varphi_{nq+s}(z;\mu)}=f_s(z),\qquad R<|z|\leq\infty.
\]
\end{theorem}

\vspace{4mm}

To clarify how our results fit into what is known about orthogonal polynomials corresponding to measures supported on polynomial lemniscates, we will immediately turn our attention to proving Theorem \ref{app} so that we may establish Corollary \ref{areap}.  After that we will consider some additional examples.  The examples in Section \ref{eg} highlight the utility and some subtleties of the results stated in this section.  Finally, in Section \ref{equiv}, we will provide proofs of the main results.

\vspace{2mm}

\noindent\textbf{Acknowledgements.}  It is a pleasure to thank Ed Saff for encouraging me to pursue this line of investigation and for much useful discussion.  I would also like to thank Barry Simon for useful feedback concerning this work.

\section{Proof of Theorem \ref{app}}\label{proofapp}

Our goal in this section is to prove Theorem \ref{app} and also Corollary \ref{areap}.  We will be considering finite measures $\mu$ supported on $\{z:|P(z)|\leq r\}$ for some monic polynomial $P$ of degree $q\geq1$ and some $r>0$ that also satisfy the bound (\ref{kbound}).  For convenience, let us define $\mu_q=|P|^2\mu$.  If $s<r$, the bound (\ref{kbound}) means
\begin{align}\label{toz}
\sum_{n=0}^{\infty}\int_{\{|P|\leq s\}}|\varphi_{n}(z;\mu_q)|^2d\mu_q(z)=\int_{\{|P|\leq s\}}\sum_{n=0}^{\infty}|\varphi_{n}(z;\mu_q)|^2d\mu_q(z)<\mu(\bbC)\cdot C_{\{z:|P(z)|\leq s\}}.
\end{align}
From this, we easily deduce that if $s<r$ is chosen arbitrarily, then
\[
\|\Phi_n(z;\mu_q)\|^2_{L^2(\mu_q)}=\left(1+o(1)\right)\int_{\{s<|P|\leq r\}}|\Phi_n(z;\mu_q)|^2d\mu_q(z)
\]
as $\nri$.

If we define (for any compactly supported measure $\nu$)
\[
\lambda(z;\nu)=\inf\left\{\int|Q(w)|^2d\nu(w):Q\mbox{ is a polynomial},\, Q(z)=1\right\},
\]
then it is well-known that
\[
\lambda(z;\nu)^{-1}=\sum_{n=0}^{\infty}|\varphi_n(z;\nu)|^2,
\]
(see \cite[Proposition 2.16.2]{Rice}).  Therefore, the bound (\ref{kbound}) is equivalent to the statement that $\lambda(z;\mu_q)$ is bounded uniformly from below away from zero on compact subsets of $\{z:|P(z)|<r\}$.  This easily implies that the same is true for $\mu$, which means for each compact set $K\subseteq\{z:|P(z)|<r\}$ there is a constant $C_K'$ so that
\[
\sum_{n=0}^{\infty}|\varphi_n(z;\mu)|^2<C_K',\qquad z\in K.
\]
Consequently, we may apply the above reasoning to $\mu$ to see that
\[
\|\Phi_n(z;\mu)\|^2_{L^2(\mu)}=\left(1+o(1)\right)\int_{\{s<|P|\leq r\}}|\Phi_n(z;\mu)|^2d\mu(z)
\]
as $\nri$.  We have therefore proven the following result:

\begin{prop}\label{boundwam}
Let $\mu$ be a measure on  $\{z:|P(z)|\leq r\}$ that satisfies the condition (\ref{kbound}).  Every weak asymptotic measure of $\mu$ is supported on $\{z:|P(z)|=r\}$.
\end{prop}

We also make the following observation, which will be useful for the main calculation of the proof.

\begin{lemma}\label{muvan}
Let $\mu$ be a measure on  $\{z:|P(z)|\leq r\}$ that satisfies the condition (\ref{kbound}).  For any $s<r$ we have
\begin{align}\label{inttozero}
\lim_{\nri}\int_{\{z:|P|\leq s\}}|\varphi_n(z;\mu_q)|^2d\mu(z)=0.
\end{align}
\end{lemma}

\begin{proof}
Given any $\epsilon>0$, we may choose $t\in(0,s)$ so that
\begin{align*}
C_{\{z:|P(z)|\leq s\}}\cdot\mu\left(\{z:0<|P(z)|\leq t\}\right)<\epsilon.
\end{align*}
We calculate
\begin{align*}
&\int_{\{z:|P|\leq s\}}|\varphi_n(z;\mu_q)|^2d\mu(z)=\int_{\{z:|P|\leq t\}}|\varphi_n(z;\mu_q)|^2d\mu(z)+\int_{\{z:t<|P|\leq s\}}|\varphi_n(z;\mu_q)|^2d\mu(z)\\
&\qquad\leq\sum_{x\in\{P^{-1}(0)\}}\mu(\{x\})|\varphi_n(x;\mu_q)|^2+\int_{\{0<|P|\leq t\}}|\varphi_n(z;\mu_q)|^2d\mu+\frac{1}{t^{2}}\int_{\{t<|P|\leq s\}}|\varphi_n(z;\mu_q)|^2d\mu_q.
\end{align*}
The first term tends to zero as $\nri$ since $\{\varphi_n(x;\mu_q)\}_{n\geq0}$ is square summable for each $x\in\{P^{-1}(0)\}$.  Equation (\ref{toz}) implies that the third term tends to zero as $\nri$.  We also notice that
\[
\int_{\{z:0<|P|\leq t\}}|\varphi_n(z;\mu_q)|^2d\mu(z)\leq C_{\{z:|P(z)|\leq s\}}\cdot\mu\left(\{z:0<|P(z)|\leq t\}\right)<\epsilon.
\]
Therefore,
\[
\limsup_{\nri}\int_{\{z:|P|\leq s\}}|\varphi_n(z;\mu_q)|^2d\mu(z)<\epsilon.
\]
Since $\epsilon>0$ was arbitrary, this establishes (\ref{inttozero}).
\end{proof}

Now we have the tools we need for our main calculation.  Indeed, if $\mu$ is supported on $\{z:|P(z)|\leq r\}$, then
\begin{align*}
\kappa_n^{-2}(\mu_q)&=\int|\Phi_n(z;\mu_q)|^2d\mu_q(z)\\
&=\int_{\{s<|P|\leq r\}}|\Phi_n(z;\mu_q)|^2|P(z)|^2d\mu(z)\left(1+o(1)\right)\\
&\geq s^2\int_{\{s<|P|\leq r\}}|\Phi_n(z;\mu_q)|^2d\mu(z)\left(1+o(1)\right)\\
&=s^2\int_{\{z:|P(z|\leq r\}}|\Phi_n(z;\mu_q)|^2d\mu(z)\left(1+o(1)\right)-s^2\int_{\{z:|P|\leq s\}}|\Phi_n(z;\mu_q)|^2d\mu(z)\left(1+o(1)\right)\\
&\geq s^2\int_{\{|P|\leq r\}}|\Phi_n(z;\mu)|^2d\mu(z)\left(1+o(1)\right)-s^2\kappa_n^{-2}(\mu_q)\int_{\{|P|\leq s\}}|\varphi_n(z;\mu_q)|^2d\mu(z)\left(1+o(1)\right)\\
&=s^2\kappa_n^{-2}(\mu)+o(\kappa_n^{-2}(\mu_q)),
\end{align*}
where we used (\ref{inttozero}) in the last equality.  Since $s<r$ was chosen arbitrarily, we conclude that
\begin{align}\label{info}
\liminf_{\nri}\frac{\kappa_n(\mu)}{\kappa_n(\mu_q)}\geq r.
\end{align}

Notice that the measure $\mu_q$ can be obtained from the measure $\mu$ by $q$ applications of the \textit{Christoffel Transform}, which is the multiplication of a measure by the square modulus of a first degree monic polynomial.  Let us write
\[
P(z)=\prod_{j=1}^q(z-x_j)
\]
and define
\[
\mu_m=\prod_{j=1}^m|z-x_j|^2\mu,\qquad m=0,1,\ldots,q,
\]
(so that $\mu_0=\mu$).  Our assumptions imply that $\lambda(z;\mu_m)$ is uniformly bounded from below away from $0$ on compact subsets of $\{z:|P(z)|<r\}$ for every $m\in\{0,1,\ldots,q\}$.  This implies that for every $m\in\{0,\ldots,q-1\}$ we have
\[
\lim_{\nri}\frac{|\varphi_n(x_{m+1};\mu_m)|^2}{\sum_{j=0}^{n-1}|\varphi_j(x_{m+1};\mu_m)|^2}=0.
\]
It then follows from the proof of \cite[Theorem 5.2]{SimaRat} that
\[
\kappa_n(\mu_{m+1})=\kappa_{n+1}(\mu_m)\left(1+o(1)\right),\qquad m\in\{0,\ldots,q-1\},
\]
as $\nri$.  Iterating this relation $q$ times and applying (\ref{info}), we conclude
\begin{align*}
\liminf_{\nri}\frac{\kappa_n(\mu)}{\kappa_{n+q}(\mu)}\geq r.
\end{align*}
However, the extremal property implies
\[
\frac{\kappa_n(\mu)}{\kappa_{n+q}(\mu)}\leq r,
\]
so we conclude
\begin{align*}
\lim_{\nri}\frac{\kappa_n(\mu)}{\kappa_{n+q}(\mu)}=r.
\end{align*}
Therefore, we may apply Theorem \ref{magicish} to conclude that for such measures $\mu$, the operators $\{(P(M)-r\mcr^q)\mcr^n\}_{n\in\bbN}$ converge strongly to $0$ as $\nri$, which completes the proof of Theorem \ref{app}.  Furthermore, Corollary \ref{prl} implies that for every $z$ not in the convex hull of the support of $\mu$, one has
\begin{align}\label{prat}
\lim_{\nri}\frac{\Phi_{n+q}(z;\mu)}{\Phi_n(z;\mu)}=P(z).
\end{align}

\begin{proof}[Proof of Corollary \ref{areap}:]
Our above calculations show that we need only prove that area measure on a polynomial lemniscate satisfies the condition (\ref{kbound}).  Therefore, we must show that $\lambda(z;|P|^2dA)$ is bounded uniformly from below away from zero on compact subsets of $\{z:|P(z)|<r\}$.

Let $u\in\{z:|P(z)<r\}$ be fixed and let $Q$ be any polynomial satisfying $Q(u)=1$.  Fix some $\delta>0$ that is small enough so that the diameter of each connected component of the set $\{z:|P(z)|<\delta\}$ is small compared to the distance between $\{P^{-1}(0)\}$ and $\{z:|P(z)|=r\}$.  First we will consider the case in which $u\in\{z:|P(z)|<\delta\}$.  In this case, our choice of $\delta$ implies that there are positive numbers $D_1$ and $D_2$ that are independent of $u$ so that
\[
\{w:|w-u|=s\}\subseteq\{z:\delta<|P(z)|<r\},\qquad s\in[D_1,D_2].
\]
Therefore, we calculate
\begin{align*}
\int|Q(z)|^2|P(z)|^2dA(z)&\geq\delta^2\int_{D_1}^{D_2}\int_0^{2\pi}|Q(u+t\eitheta)|^2d\theta\,tdt\\
&\geq2\pi\delta^2\int_{D_1}^{D_2}\left|\int_0^{2\pi}Q(u+t\eitheta)^2\frac{d\theta}{2\pi}\right|tdt\\
&=\pi\delta^2(D_2^2-D_1^2).
\end{align*}
Now we will consider the case in which $u\subseteq\{z:\delta\leq|P(z)|<r\}$.  Let $d$ be the distance from $u$ to the boundary of the lemniscate.  In this case, we calculate
\begin{align*}
\int|Q(z)|^2|P(z)|^2dA(z)&\geq\int_{0}^{d}\int_0^{2\pi}|Q(u+t\eitheta)P(u+t\eitheta)|^2d\theta\,tdt\\
&\geq2\pi\int_{0}^{d}\left|\int_0^{2\pi}Q(u+t\eitheta)^2P(u+t\eitheta)^2\frac{d\theta}{2\pi}\right|tdt\\
&=\pi d^2|P(u)|^2\\
&\geq\pi\delta^2 d^2.
\end{align*}
It follows that if $u\in\{z:|P(z)|<r\}$, then
\[
\lambda(u,|P|^2\,dA)\geq\min\left\{\pi\delta^2(D_2^2-D_1^2),\pi\delta^2d^2\right\}.
\]
This implies area measure on the lemniscate $\{z:|P(z)|\leq r\}$ satisfies the condition (\ref{kbound}) as desired.
\end{proof}

\section{Further Examples}\label{eg}

Theorem \ref{app} provides a large class of examples that satisfy the conditions of Theorem \ref{magicish}.  Our focus in this section is a discussion of additional examples.  We have already seen that area measure on a polynomial lemniscate is a measure to which we can apply Theorem \ref{app}.  The calculations in Section \ref{proofapp} show that if we define $\mu=\sigma+\nu$, where $\sigma$ is area measure on $\{z:|P(z)|\leq r\}$ and $\nu$ is any positive and finite measure supported on $\{z:|P(z)|\leq r\}$, then the measure $\mu$ satisfies the hypotheses of Theorem \ref{app}.  Here is another example of a collection of measures to which our results apply.

\vspace{2mm}

\noindent\textbf{Example.}  In this example, we will consider measures supported on the boundary of a lemniscate region.  Let $\mu$ be supported on
\[
E=\bigcup_{i=1}^pE_i,
\]
where the collection $\{E_i\}_{i=1}^p$ consists of $p$ mutually exterior disjoint Jordan curves and each $E_i$ is a connected component of the set $\{z:|P(z)|=r\}$ for some monic polynomial $P$ of degree $q\geq p$ and some $r>0$.  Furthermore, assume that on each component $E_i$, the measure $\mu$ is absolutely continuous with respect to arc-length measure and the derivative $\mu'$ is continuous and bounded from below by a strictly positive constant.  The orthogonal polynomials for such measures were studied extensively in \cite{Widom}.  It is clear that for such a measure, every weak asymptotic measure is supported on $\{z:|P(z)|=r\}$.  Also, one can briefly adapt the calculations in Section \ref{proofapp} to show that $\lim_{\nri}\kappa_n\kappa_{n+q}^{-1}=r$ (see (\ref{lam2}) below).  We conclude that the operators $\{(P(M)-r\mcr^q)\mcr^n\}_{n\in\bbN}$ converge strongly to zero as $\nri$.

Now we will consider a perturbation of $\mu$ in the form of a mass point interior to one of the curves $E_i$.  We claim that $\lambda(\cdot;\mu)>0$ at every point interior to any of the curves $E_i$.  To see this, notice that since the components of $E$ are disjoint, they are smooth.  Therefore, if $x$ is interior to $E_i$, the harmonic measure for $E_i$ and the point $x$ (call it $\omega_{i,x}$) is absolutely continuous with respect to arc-length measure and has continuous derivative bounded above and below by positive constants (see \cite[page 65]{GarnMar}).  It follows that if $Q$ is a polynomial and $Q(x)=1$, then there is a positive constant $C$ so that
\begin{align}\label{lam2}
\|Q\|^2_{|L^2(\mu)}\geq\int_{E_i}|Q(z)|^2\mu'(z)d|z|\geq C\int_{E_i}|Q(z)|^2d\omega_{i,x}(z)
\geq C\left|\int_{E_i}Q(z)^2d\omega_{i,x}(z)\right|=C.
\end{align}

We conclude that $\{\varphi_n(x;\mu)\}_{n\in\bbN}\in\ell^2(\bbN)$ (as in the previous section).  This immediately implies
\[
\lim_{\nri}\frac{|\varphi_n(x;\mu)|^2}{\sum_{j=0}^{n-1}|\varphi(x;\mu)|^2}=0,
\]
and hence we may apply the results of \cite{SimaRat} to conclude that for any $t>0$ it holds that
\begin{align*}
\lim_{\nri}\frac{\kappa_n(\mu)}{\kappa_{n}(\mu+t\delta_x)}=1,
\end{align*}
which means
\[
\lim_{\nri}\frac{\kappa_n(\mu+t\delta_x)}{\kappa_{n+q}(\mu+t\delta_x)}=r.
\]
Even after adding this mass point to the measure, it is still true that $\lambda(\cdot;\mu+t\delta_x)$ is strictly positive on the interior of the curves comprising $E$.  Therefore, we may repeat the above procedure to add finitely many mass points interior to the curves $\{E_i\}_{i=1}^p$.  We conclude that for such perturbed measures, the operators $\{(P(M)-r\mcr^q)\mcr^n\}_{n\in\bbN}$ converge strongly to zero as $\nri$.  Additionally, Corollary \ref{prl} implies every right limit of $M$ is $q$-periodic along its diagonals and we have ratio asymptotics as in (\ref{prat}).

\vspace{2mm}

One of the subtleties of Theorem \ref{magicish} is that it does not require any a priori knowledge that $\mu$ is supported on a polynomial lemniscate (while Theorem \ref{app} does).  Indeed, our next example shows that Theorem \ref{magicish} applies in some cases when the measure $\mu$ has support outside such a lemniscate.

\vspace{2mm}

\noindent\textbf{Example.}  In this example, we will consider a measure supported on a collection of mutually exterior disjoint Jordan curves and a finite collection of points that lie exterior to all of these curves.  The orthogonal polynomials for such measures were studied in \cite{KaliKon}, and we will rely heavily on those results in this example.  As in the previous example, let
\[
E=\bigcup_{i=1}^pE_i,
\]
where the collection $\{E_i\}_{i=1}^p$ consists of $p$ mutually exterior disjoint Jordan curves and each $E_i$ is a connected component of the set $\{z:|P(z)|=r\}$ for some monic polynomial $P$ of degree $q\geq p$ and some $r>0$.  We will consider measures $\mu$ that are supported on $E\cup\{z_j\}_{j=1}^N$, where
\[
\{z_j\}_{j=1}^N\subseteq\{z:|P(z)|>r\}.
\]
Assume further that on each component $E_i$, the measure $\mu$ is absolutely continuous with respect to arc-length measure and the derivative $\mu'$ is continuous and bounded from below by a strictly positive constant $\sigma$.

The capacity of the set $E$ is $\sqrt[q]{r}$ (see \cite[Equation III.3.7]{SaffTot}), so let us define the measure $\nu$ by
\[
d\nu(\zeta)=\left(\prod_{j=1}^N|\zeta-z_j|^2/\sqrt[q]{r^2}\right)d\mu(\zeta).
\]
The previous example tells us that
\begin{align}\label{nas}
\lim_{\nri}\frac{\kappa_n(\nu)}{\kappa_{n+q}(\nu)}=r.
\end{align}
Combining \cite[Lemma 2]{KaliKon} and \cite[Theorem 2]{KaliKon}, we deduce that
\begin{align}\label{mnas}
\lim_{\nri}\frac{\kappa_{n-N}(\nu)}{\kappa_n(\mu)}=\sqrt[q]{r^{N}}.
\end{align}
If we combine equations (\ref{nas}) and (\ref{mnas}), we conclude that
\[
\lim_{\nri}\frac{\kappa_n(\mu)}{\kappa_{n+q}(\mu)}=r.
\]

Now we must prove that every weak asymptotic measure is supported on E, but this is not difficult. Indeed, it is clear that $\lambda(z_i;\mu)\geq1$ for all $i\in\{1,\ldots,N\}$, so
\[
\lim_{\nri}\sum_{j=1}^N\mu(\{z_j\})|\varphi_n(z_j;\mu)|^2=0,
\]
which immediately implies every weak asymptotic measure is supported on $E$. Now we may apply Theorem \ref{magicish} to conclude that the operators $\{(P(M)-r\mcr^q)\mcr^n\}_{n\in\bbN}$ converge strongly to $0$ as $\nri$.  Corollary \ref{prl} then implies that for every $z$ not in the convex hull of the support of $\mu$ it is true that
\begin{align*}
\lim_{\nri}\frac{P(z)\varphi_n(z;\mu)}{r\varphi_{n+q}(z;\mu)}=1,
\end{align*}
and every right limit of $M$ is $q$-periodic along its diagonals.

\vspace{2mm}

The next example provides a specific measure to which we can apply Theorem \ref{uniquer}.

\vspace{2mm}

\noindent\textbf{Example.}  If $\mu$ is area measure on the region defined by $\{z:|z^q-1|<r\}$ for some $r<1$, then \cite[Proposition 7.1]{Islands} implies
\[
\lim_{\nri}\kappa_{nq+s}\kappa_{nq+s+1}^{-1}=
\begin{cases}
1, & \qquad\, s\in\{0,\ldots,q-2\}\\
 r, & \qquad\, s=q-1
\end{cases}.
\]
Furthermore, \cite[Proposition 7.3]{Islands} also tells us that if $|z|$ is sufficiently large, then
\begin{align*}
\lim_{\nri}\frac{\Phi_{nq+s}(z;\mu)}{\Phi_{nq+s+1}(z;\mu)}=\begin{cases}
\frac{1}{z}\left(\frac{z^q-1+r^{2}}{z^q-1}\right)^{1/q},\quad &s\in\{0,1,\ldots,q-2\}\\
\frac{z^{q-1}}{z^q-1}\left(\frac{z^q-1+r^{2}}{z^q-1}\right)^{(1-q)/q},\quad &s=q-1.
\end{cases}
\end{align*}
Combining these two results and invoking Theorem \ref{uniquer} allows us to conclude that in this case, the matrix $M$ is asymptotically $q$-block Toeplitz.

\vspace{2mm}

The next section is devoted to the proofs of the main results.

\section{Proof of the Main Theorems}\label{equiv}

This section is devoted to the proofs of the main results from Section \ref{intro}.  We will say that a matrix $A$ is \textit{banded of width $m$} if $A_{j,k}=0$ whenever $|j-k|>m$.  We will denote by $e_k$, the element of $\ell^2(\bbN)$ with a $1$ in position $k$ and zeros elsewhere.

Our first task is to prove Theorem \ref{magicish}, the first step of which is the following lemma.

\begin{lemma}\label{upr}
Let $f$ be an entire function of $z$ and let $\mcg$ be a banded Toeplitz matrix of width $q\geq1$.  Suppose the operators $\{(f(M)-\mcg)\mcr^n\}_{n\in\bbN}$ converge strongly to $0$ as $\nri$.  Suppose further that
\[
\limsup_{\nri}\left(\|\mcg^ne_{(n+3)q}\|\right)^{1/n}=r.
\]
Then every weak asymptotic measure $\gamma$ is supported on $\{z:|f(z)|\leq r\}$.  If in fact
\begin{align}\label{actulim}
\lim_{\nri}\left(\|\mcg^ne_{(n+3)q}\|\right)^{1/n}=r,
\end{align}
then it is also true that every weak asymptotic measure $\gamma$ satisfies
\[
\{z:|f(z)|=r\}\cap\supp(\gamma)\neq\emptyset.
\]
\end{lemma}

\begin{proof}
Notice that
\[
\|(f(M)^k-\mcg^k)e_{n+1}\|^2=\|(f(M)^k-\mcg^k)\mcr^ne_1\|^2\rightarrow0,\qquad\nri.
\]
We conclude that for every $k\in\bbN$, the following limit relation holds:
\begin{align}\label{fabslim}
\lim_{\nri}\|f(M)^ke_n\|^2=\lim_{\nri}\|\mcg^ke_n\|^2=\|\mcg^ke_{(k+3)q}\|^2,
\end{align}
since $\mcg$ is a banded Toeplitz matrix of width $q$.  Our hypotheses imply that if $\varepsilon>0$ is arbitrary then we can find an $n_0\in\bbN$ that is sufficiently large so that $\|\mcg^ke_{(k+3)q}\|^2<(r+\varepsilon)^{2k}$ whenever $k>n_0$.

Now, let $\gamma$ be a weak asymptotic measure and $\mcn$ the corresponding subsequence.  Suppose for contradiction that there is some number $t>r$ so that $\gamma(\{z:|f(z)|\geq t\})=\beta>0$.  Then for all sufficiently large values of $k$ it holds true that
\begin{align*}
\left(r+\frac{t-r}{2}\right)^{2k}&>\limsup_{{\nri}\atop{n\in\mcn}}\|f(M)^ke_{n+1}\|^2=\lim_{{\nri}\atop{n\in\mcn}}\int|f(z)|^{2k}|\varphi_n(z;\mu)|^2d\mu(z)\\
&=\int|f(z)|^{2k}d\gamma(z)\geq t^{2k}\beta.
\end{align*}
Clearly this gives a contradiction when $k$ is sufficiently large.  We conclude that $\supp(\gamma)\subseteq\{z:|f(z)|\leq r\}$.

If we make the stronger assumption (\ref{actulim}), then if $\varepsilon\in(0,r)$ is arbitrary we can find an $n_1\in\bbN$ so that $\|\mcg^ke_{(k+3)q}\|^2>(r-\varepsilon)^{2k}$ whenever $k>n_1$.  Again, let $\gamma$ be a weak asymptotic measure with corresponding subsequence $\mcn$.  Suppose for contradiction that there is a positive $s<r$ so that $\gamma(\{z:|f(z)|\leq s\})=1$.  Then for all sufficiently large values of $k$ it holds true that
\begin{align*}
\left(r-\frac{r-s}{2}\right)^{2k}&<\liminf_{{\nri}\atop{n\in\mcn}}\|f(M)^ke_{n+1}\|^2=\lim_{{\nri}\atop{n\in\mcn}}\int|f(z)|^{2k}|\varphi_n(z;\mu)|^2d\mu(z)\\
&=\int|f(z)|^{2k}d\gamma(z)\leq s^{2k}.
\end{align*}
Clearly this gives a contradiction for all $k>n_1$.  We conclude that $\supp(\gamma)\cap\{z:|f(z)|=r\}$ is non-empty.
\end{proof}

Consider briefly the case when $\mu$ has compact support in the real line.  If $P$ is a monic polynomial of degree $q\geq1$ and $P(M)$ has unique right limit given by $\mcl^q+\mcr^q$ (on $\ell^2(\bbZ)$), then a straightforward application of Lemma \ref{upr} shows that every weak asymptotic measure is supported on $\{z:|P(z)|\leq2\}$ and every weak asymptotic measure contains in its support at least one point where $|P(z)|=2$.  

\vspace{2mm}

Now we turn to the proof of Theorem \ref{magicish}.

\begin{proof}[Proof of Theorem \ref{magicish}]
Let us begin by assuming both properties (i) and (ii).  We calculate
\begin{align*}
\|(P(M)-r\mcr^q)e_n\|^2&=\|P\varphi_{n-1}-r\varphi_{n+q-1}\|^2\\
&=\|P\varphi_{n-1}\|^2+r^2-2r\frac{\kappa_{n-1}}{\kappa_{n+q-1}}\\
&\rightarrow0,
\end{align*}
by our assumptions (i) and (ii).  It follows that
\[
\lim_{\nri}\|(P(M)-r\mcr^q)\mcr^ne_j\|=0,\qquad j\in\bbN,
\]
which implies the desired strong convergence.

To prove the converse, let us assume that the operators $\{(P(M)-r\mcr^q)\mcr^n\}_{n\in\bbN}$ converge strongly to $0$ as $\nri$.  First notice that this implies $\mcl^nP(M)\mcr^n$ converges to $r\mcr^q$ weakly as $\nri$, so it is obvious that $\kappa_n\kappa_{n+q}^{-1}\rightarrow r$ as $\nri$.

Now, let $\gamma$ be any weak asymptotic measure with corresponding subsequence $\mcn$.  Our strong convergence hypothesis and Lemma \ref{upr} imply that $\gamma$ is supported on $\{z:|P(z)|\leq r\}$.  Suppose for contradiction that there is an $s<r$ so that $\gamma(\{z:|P(z)|\leq s\})=\beta>0$.  The relation (\ref{fabslim}) implies
\[
\lim_{\nri}\|P(M)^{k}e_{n+1}\|^2=r^{2k},\qquad k\in\bbN.
\]
Therefore, we calculate
\begin{align*}
r^{2k}=\lim_{{\nri}\atop{n\in\mcn}}\|P(M)^{k}e_{n+1}\|^2&=\lim_{{\nri}\atop{n\in\mcn}}\int|P(z)|^{2k}|\varphi_n(z;\mu)|^2d\mu(z)\\
&=\int|P(z)|^{2k}d\gamma(z)\leq s^{2k}\beta+r^{2k}(1-\beta).
\end{align*}
This is a contradiction for every $k$, which implies $\gamma$ is supported on $\{z:|P(z)|= r\}$ as desired.
\end{proof}

Theorem \ref{magicish} lead us to Corollary \ref{prl}, which now enables us to prove Corollary \ref{limittr}.

\begin{proof}[Proof of Corollary \ref{limittr}:]
The relation (\ref{keydet}) and Cramer's Rule allow us to write
\begin{align*}
\frac{\Phi_{n-1}(z;\mu)}{\Phi_n(z;\mu)}=\left((z-\pi_nM\pi_n)^{-1}\right)_{n,n}=\sum_{j=0}^{\infty}\frac{((\pi_nM\pi_n)^j)_{n,n}}{z^{j+1}},\qquad|z|>\|\mcm\|.
\end{align*}
Corollary \ref{prl} tells us that if $|z|$ is sufficiently large, then
\begin{align*}
\frac{1}{P(z)}&=\lim_{\nri}\frac{\Phi_n(z;\mu)}{\Phi_{n+q}(z;\mu)}=
\lim_{\nri}\prod_{k=1}^q\frac{\Phi_{n+k-1}(z;\mu)}{\Phi_{n+k}(z;\mu)}\\
&=\lim_{\nri}\prod_{k=1}^q\left(\sum_{j=0}^{\infty}\frac{\left((\pi_{n+k}M\pi_{n+k})^j\right)_{n+k,n+k}}{z^{j+1}}\right).
\end{align*}
If the zeros of $P$ are $\{w_1,\ldots,w_q\}$, then $\alpha=\sum w_j$ and we can write
\[
\lim_{\nri}\prod_{k=1}^q\left(\sum_{j=0}^{\infty}\frac{(\pi_{n+k}M\pi_{n+k})^j_{n+k,n+k}}{z^{j+1}}\right)=\frac{1}{P(z)}=\frac{1}{z^q}\prod_{k=1}^q\left(\sum_{j=0}^{\infty}\frac{w_k^j}{z^j}\right).
\]
The desired conclusion follows by equating coefficients of $z^{-q-1}$.
\end{proof}

Let us now turn our attention to the proof of Theorem \ref{uniquer}.

\begin{proof}[Proof of Theorem \ref{uniquer}]
Let $\mu$ be as stated in the theorem and suppose that for every $s\in\{0,1,\ldots,q-1\}$ there is a a positive number $R$ and a function $f_s$ that is analytic in $\{z:R<|z|\leq\infty\}$ so that
\[
\lim_{\nri}\frac{\varphi_{nq+s-1}(z;\mu)}{\varphi_{nq+s}(z;\mu)}=f_s(z),\qquad R<|z|\leq\infty.
\]
In particular, we then know that for each $s\in\{0,\ldots,q-1\}$, the limit
\begin{align}\label{plusone}
\lim_{\nri}\frac{\kappa_{nq+s-1}}{\kappa_{nq+s}}=\lim_{\nri}M_{nq+s+1,nq+s}
\end{align}
exists.  Since we are assuming that $\liminf_{\nri}\kappa_{n-1}\kappa_{n}^{-1}>0$, we know that the limits in (\ref{plusone}) are all non-zero, so the monic orthogonal polynomials also exhibit ratio asymptotics through the sequences of the form $\{nq+s\}_{n\in\bbN}$.  The relation (\ref{keydet}) and Cramer's Rule allow us to write
\begin{align}\label{ratiosum}
\frac{\Phi_{n-1}(z;\mu)}{\Phi_n(z;\mu)}=\left((z-\pi_nM\pi_n)^{-1}\right)_{n,n}=\sum_{j=0}^{\infty}\frac{((\pi_nM\pi_n)^j)_{n,n}}{z^{j+1}},\qquad|z|>\|\mcm\|.
\end{align}
Therefore, for every $j\geq1$ and $s\in\{0,\ldots,q-1\}$, the following limit exists:
\begin{align}\label{allj}
\lim_{\nri}\left((\pi_{nq+s}M\pi_{nq+s})^j\right)_{nq+s,nq+s}.
\end{align}
The remainder of the proof is by induction.  For the base case, we apply (\ref{allj}) with $j=1$ to see that $\lim_{\nri}M_{nq+s,nq+s}$ exists for every $s\in\{0,\ldots,q-1\}$.

Suppose for our induction hypothesis that for a given $k\in\bbN$ there are numbers $\{A_{j,s}\}_{j=-1,s=0}^{j=k-1,s=q-1}$ so that
\begin{align*}
\lim_{\nri}M_{nq+s-j,nq+s}=A_{j,s}.
\end{align*}
Since $M$ is a Hessenberg matrix, we know
\begin{align}\label{ssum}
\nonumber&\lim_{\nri}\left((\pi_{nq+s}M\pi_{nq+s})^{k+1}\right)_{nq+s,nq+s}\\
&\qquad\qquad=\lim_{\nri}\sum_{i_1,\ldots,i_{k}=0}^{k}
M_{nq+s,nq+s-i_1}M_{nq+s-i_1,nq+s-i_2}\cdots M_{nq+s-i_k,nq+s}.
\end{align}
The induction hypothesis implies (setting $i_0=i_{k+1}=0$)
\[
\lim_{\nri}M_{nq+s-i_{p-1},nq+s-i_p}
\]
exists for every $p\in\{1,\ldots,k+1\}$ provided $i_{p-1}-i_{p}\leq k-1$.  The Hessenberg structure of the matrix $M$ assures us that $i_{p-1}-i_p\geq-1$ for every $p\in\{1,\ldots,k\}$ in order for the corresponding term in the sum (\ref{ssum}) to be non-zero.  Therefore, the induction hypothesis implies every term in the sum (\ref{ssum}) converges as $\nri$ except for the one in which $i_m=m$ for every $m\in\{1,\ldots,k\}$.  However, since the limit on the left hand side of (\ref{ssum}) exists, we conclude that the following limit exists:
\[
\lim_{\nri}M_{nq+s,nq+s-1}M_{nq+s-1,nq+s-2}\cdots M_{nq+s-k,nq+s}
\]
for every $s\in\{0,\ldots,q-1\}$.  By (\ref{plusone}), this means that for each $s\in\{0,\ldots,q-1\}$, there must be a number $A_{k,s}$ so that
\[
\lim_{\nri}M_{nq+s-k,nq+s}=A_{k,s}.
\]
This completes the induction and implies our desired conclusion.

For the converse, suppose that $M$ is asymptotically $q$-block Toeplitz.  Our hypotheses imply that for every $m,k\in\bbZ$ and $s\in\{0,\ldots,q-1\}$ the following limit exists:
\[
\lim_{\nri}M_{nq+s+m,nq+s+k}.
\]
In particular, this implies
\[
\lim_{\nri}\left((\pi_{nq+s}M\pi_{nq+s})^j\right)_{nq+s,nq+s}
\]
exists for every $j\in\bbN$ (by the formula (\ref{ssum})).  It then follows from (\ref{ratiosum}) that there are functions $\{g_s\}_{s=0}^{q-1}$ that are analytic in $\{z:\|\mcm\|<|z|\leq\infty\}$ such that
\begin{align}\label{gsexist}
\lim_{\nri}\frac{\Phi_{nq+s-1}(z;\mu)}{\Phi_{nq+s}(z;\mu)}=g_s(z),\qquad\|\mcm\|<|z|\leq\infty.
\end{align}
Furthermore, since $M$ is asymptotically $q$-block Toeplitz, it must be the case that for each $s\in\{0,\ldots,q-1\}$ the limit $\lim_{\nri}\kappa_{nq+s-1}\kappa_{nq+s}^{-1}$ exists and is non-zero (since $\liminf_{\nri}\kappa_{n-1}\kappa_{n}^{-1}>0$) for every $s\in\{0,\ldots,q-1\}$.  Combining this with (\ref{gsexist}) proves the existence of functions $\{f_s\}_{s=0}^{q-1}$ as in the statement of the theorem.
\end{proof}

\vspace{7mm}

\noindent \textsc{Brian Simanek,  Department of Mathematics, 1326 Stevenson Center, Vanderbilt University, Nashville, TN}

\vspace{1mm}

\noindent \texttt{brian.z.simanek$\at$vanderbilt.edu}


\vspace{7mm}


\begin{thebibliography}{99}



\bibitem{LBS} L. Baratchart, E. B. Saff, and N. Stylianopoulos, {\em On finite-term recurrence relations for Bergman and Szeg\H{o} polynomials}, Comp. Methods Func. Theory 12, (2012), no. 2, 393--402.

\bibitem{MagicForm} D. Damanik, R. Killip, and B. Simon, {\em Perturbations of orthogonal polynomials with periodic recursion coefficients}, Annals of Math. 171 (2010), 1931--2010.

\bibitem{GarnMar} J. Garnett and D. Marshall, {\em Harmonic Measure}, Cambridge University Press, Cambridge, 2005.

\bibitem{Islands} B. Gustafsson, M. Putinar, E. B. Saff, and N. Stylianopoulos, {\em Bergman polynomials on an Archipelago: estimates, zeros, and shape reconstruction}, Advances in Mathematics 222 (2009), 1405--1460.

\bibitem{KaliKon} V. Kaliaguine and A. Kononova, {\em On the asymptotics of polynomials orthogonal on a system of curves with respect to a measure with discrete part}, St. Petersburg Math J. 21 (2010), no. 2, 217--230.



\bibitem{PuSty} M. Putinar and N. Stylianopoulos, {\em Finite-term relations for planar orthogonal polynomials}, Complex Anal. Oper. Theory 1 (2007), no. 3, 447--456.

\bibitem{SSST} E. B. Saff, H. Stahl, N. Stylianopoulos, and V. Totik, {\em Orthogonal polynomials for area-type measures and image recovery}, preprint available at arxiv.org/abs/1043.6456.

\bibitem{Shift} E. B. Saff and N. Stylianopoulos, {\em Asymptotics for Hessenberg matrices for the Bergman shift operator on Jordan regions}, Complex Anal. Oper Theory 8 (2014), no. 1, 1--24.

\bibitem{SaffTot} E. B. Saff and V. Totik {\em Logarithmic Potentials with External Fields}, Grundlehren der Mathematischen Wissenschaften, Band 316, Springer, Berlin-Heidelberg, 1997.


\bibitem{SimaRat} B. Simanek, {\em A new approach to ratio asymptotics for orthogonal polynomials}, Journal of Spectral Theory 2 (2012), no. 4, 373--395.

\bibitem{SimHess1} B. Simanek, {\em Ratio asymptotics, Hessenberg matrices, and weak asymptotic measures}, accepted for publication in Int. Math. Res. Not.


\bibitem{OPUC1} B. Simon, {\em Orthogonal Polynomials on the Unit Circle, Part One: Classical Theory}, American Mathematical Society, Providence, RI, 2005.

\bibitem{OPUC2} B. Simon, {\em Orthogonal Polynomials on the Unit Circle, Part Two: Spectral Theory}, American Mathematical Society, Providence, RI, 2005.


\bibitem{CDRev} B. Simon, {\em The Christoffel-Darboux kernel}, in: Perspectives in PDE, Harmonic Analysis, and Applications, A Volume in Honor of V. G. Maz'ya's 70th Birthday, Proc. Sympos. Pure Math. 79 (2008), 295--335.

\bibitem{Rice} B. Simon, {\em Szeg\H{o}'s Theorem and its Descendants: Spectral Theory for $L^2$ perturbations of Orthogonal Polynomials}, Princeton University Press, Princeton, NJ, 2010.



\bibitem{Slides} N. Stylianopoulos, {\em Bergman Polynomials on Archipelaga: Recent Developments on Theory and Applications}, plenary address given at CMFT 2013 conference, available at www2.ucy.ac.cy/nikos/CMFT2013.pdf.


\bibitem{Widom} H. Widom, {\em Extremal polynomials associated with a system of curves and arcs in the complex plane}, Adv. Math. 3 (1969), 127--232.

\end{thebibliography}
\end{document}